\documentclass{amsart}

\usepackage{pifont}
\usepackage{amsfonts}

\usepackage{amscd,amssymb,amsmath,graphicx,verbatim,mathrsfs}
\usepackage[TS1,OT1,T1]{fontenc}
\usepackage{extarrows}
\usepackage[all]{xy}
\usepackage{dcpic,mathtools,pictexwd}
\usepackage{rotating}
\theoremstyle{remark}
\usepackage{indentfirst}
\usepackage{amsthm} 
\usepackage{tikz}
\usetikzlibrary{angles,quotes}
\usetikzlibrary{calc}

\usepackage{color}
\usepackage[colorlinks=true, allcolors=blue]{hyperref}
\usepackage{cleveref}
\usepackage{soul}
\usepackage{enumitem}
\usepackage{tikz-cd}

\allowdisplaybreaks

\theoremstyle{plain}

\newtheorem{thm}{Theorem}[section]
\newtheorem{lem}[thm]{Lemma}

\theoremstyle{definition}

\newtheorem{defn}[thm]{Definition}

\newtheorem{rmk}[thm]{Remark}

\theoremstyle{plain}
\newcounter{thmintroctr}

\theoremstyle{definition}
\newcounter{goalintroctr}

\numberwithin{equation}{section}

\begin{document}

	\title[Reflexive and strictly convex Banach spaces without Property $(H)$ ]{R\lowercase{eflexive and strictly convex}  B\lowercase{anach spaces without} P\lowercase{roperty} $(H)$}

	\author{Chunliu Feng}
	\address{Department of Mathematics, Hebei Normal University, Hebei, China}
	\email{fengfcl@126.com}
	\thanks{ }

	\author{Geng Tian}
	\address{Department of Mathematics, Liaoning University, Liaoning, China}
	\email{gengtian.ncg@gmail.com}
	
	\thanks{The second author was supported in part by Grant for Excellent Young Scholars in Tianyuan Mathematics.}


\begin{abstract}
We prove that Property~$(H)$ fails for a broad class of
reflexive strictly convex  Banach spaces.
More generally, the real Banach space
$\bigl(\bigoplus_{j=1}^{\infty}
L^{p_j}(\Omega_j,\mu_j)\bigr)_{\ell^s}$
does not have Property~$(H)$ whenever $1\leq s<\infty$,
$1\leq p_j<\infty$, $\sup_j p_j=+\infty$,
and each $L^{p_j}(\Omega_j,\mu_j)$ is infinite-dimensional.
\end{abstract}

\maketitle

\section{Introduction}
	
Recently, L.~Cheng, Q.~Cheng, and Y.~Wang
\cite[Theorem~1.1]{ChengChengWang2026} proved that $c_0$ does not
have Property~$(H)$, answering a question raised by Kasparov and Yu.
Here $c_0$ denotes the real Banach space of sequences converging
to zero, equipped with the supremum norm.
Property~$(H)$ was introduced by Kasparov and Yu in their work
on the Novikov conjecture \cite{KasparovYu2012}.
They proved that a countable discrete group admitting a coarse
embedding into a Banach space with Property~$(H)$ satisfies the
strong Novikov conjecture.
This connection motivates the study of the geometric conditions
under which a Banach space admits Property~$(H)$.

In this paper, we prove that Property~$(H)$ fails for a broad
class of Banach spaces, including reflexive strictly convex
Banach spaces.
More precisely, let $1\leq s<\infty$, and let
$\{p_j\}_{j\geq1}\subset[1,\infty)$ satisfy
$\sup_j p_j=\infty$.
For each $j\geq1$, let $(\Omega_j,\Sigma_j,\mu_j)$ be a
measure space such that $L^{p_j}(\Omega_j,\mu_j)$ is
infinite-dimensional.
Our main result, Theorem~\ref{prop:lp-direct-sum-no-property-H},
shows that the real Banach space
\[
\mathscr X:=
\left(
\bigoplus_{j=1}^{\infty}L^{p_j}(\Omega_j,\mu_j)
\right)_{\ell^s}
\]
does not have Property~$(H)$.
Here the $\ell^s$-direct sum consists of all sequences
$x=(x^{(j)})_{j\geq1}$ with
$x^{(j)}\in L^{p_j}(\Omega_j,\mu_j)$ and
$\sum_{j=1}^{\infty}
\|x^{(j)}\|_{L^{p_j}(\Omega_j,\mu_j)}^s<\infty$,
equipped with the norm
\[
\|x\|_{\mathscr X}
=
\left(
\sum_{j=1}^{\infty}
\|x^{(j)}\|_{L^{p_j}(\Omega_j,\mu_j)}^s
\right)^{1/s}.
\]

If, in addition, $1<s<\infty$ and $1<p_j<\infty$ for every $j$,
then $\mathscr X$ is reflexive and strictly convex with its
$\ell^s$-direct sum norm.
Reflexivity follows from the corresponding property of the
summands and the standard theorem on $\ell^s$-sums of reflexive
spaces; see \cite[Lemma~2.6.2]{Fackler2011}.
Strict convexity follows because each summand is strictly convex
and an $\ell^s$-sum of strictly convex spaces is strictly convex
when $1<s<\infty$.
In particular,
\[
\left(
\bigoplus_{j=1}^{\infty}\ell^{2j}(\mathbb{N})
\right)_{\ell^2}
\]
is a reflexive strictly convex Banach space without
Property~$(H)$.
Consequently, reflexivity and strict convexity, even together,
do not imply Property~$(H)$.
Moreover, each  $\ell^{2j}(\mathbb{N})$  has Property~$(H)$ by the
classical Mazur map construction; see \cite{KasparovYu2012}.
Thus, for every $1\leq s<\infty$, Property~$(H)$ is not
preserved under countable $\ell^s$-direct sums.

Our proof is short and elementary, taking the existence of
regular expander families as a standard input.
It combines linear embeddings of finite-dimensional
$\ell^\infty$-spaces with uniformly bounded distortion
and a topological centering argument.
The Hilbert-valued Poincar\'e inequality then yields a
contradiction to the uniform continuity required by
Property~$(H)$.

We note that the same argument shows that the Schatten-class direct sum
\[
\left(
\bigoplus_{j=1}^{\infty}S^{p_j}(\mathscr H_j)
\right)_{\ell^s},
\]
regarded as a real Banach space, also fails to have
Property~$(H)$ under the same assumptions on $s$ and
$\{p_j\}$, where each $\mathscr H_j$ is an infinite-dimensional
Hilbert space and $S^{p_j}(\mathscr H_j)$ denotes the
corresponding Schatten class.

We now recall the definition of Property~$(H)$.
\begin{defn}
	Let $\mathscr X$ be a real Banach space, and let $S(\mathscr X)$
	denote its unit sphere.
	We say that $\mathscr X$ has \emph{Property $(H)$} if there exist
	two increasing sequences $\{\mathscr V_n\}_{n\geq1}$ and
	$\{\mathscr{W}_n\}_{n\geq1}$ 
	of finite-dimensional subspaces of
	$\mathscr X$ and a real Hilbert space $\mathscr{H}$, respectively, such that
	$\bigcup_{n\geq1}\mathscr V_n$
	is dense in $\mathscr X$, and there exists a uniformly
	continuous map
	\[
	\phi:S(\mathscr X)\longrightarrow S(\mathscr{H})
	\]
	whose restriction
	\[
	\phi_n:=\phi\big|_{S(\mathscr V_n)}:
	S(\mathscr V_n)\longrightarrow S(\mathscr{W}_n)
	\]
	is a homeomorphism for every $n\geq1$.
\end{defn}

\section{Preliminary} 

We recall some background on regular graphs and the
Hilbert-valued Poincar\'e inequality.
The following  material is standard; we refer
the reader to
\cite{HooryLinialWigderson2006,Naor2021,MarcusSpielmanSrivastava2015}
for further details.

A finite simple undirected graph is a pair
$\mathcal G=(\mathcal V,\mathcal E)$, where $\mathcal V$
is a finite nonempty set of vertices and each element of
$\mathcal E$ is an unordered pair $\{v,w\}$ of distinct
vertices, representing a single edge.
There are no loops, and any two distinct vertices are
joined by at most one edge.
Two vertices $v,w\in\mathcal V$ are called adjacent if
$\{v,w\}\in\mathcal E$.

The graph is connected if, for any two distinct vertices
$v,w\in\mathcal V$, there exists a finite sequence of vertices
\[
v=v_0,v_1,\ldots,v_m=w
\]
such that
\[
\{v_{i-1},v_i\}\in\mathcal E,
\qquad i=1,\ldots,m.
\]
For an integer $d\geq1$,  the graph is called $d$-regular
if every vertex is adjacent to exactly $d$ distinct vertices.

Let $\mathcal G=(\mathcal V,\mathcal E)$ be a finite connected
simple undirected $d$-regular graph with $d\geq1$,
and write $k:=|\mathcal V|$.
After choosing an ordering of $\mathcal V$, its adjacency
matrix $A=(A_{vw})_{v,w\in\mathcal V}$ is defined by
\[
A_{vw}:=
\begin{cases}
	1, & \{v,w\}\in\mathcal E,\\
	0, & \text{otherwise}.
\end{cases}
\]
Its normalized adjacency matrix is
\[
P:=\frac1d A.
\]
Equivalently, $P_{vw}$ is the probability of moving from
$v$ to $w$ when one of the $d$ neighbors of $v$ is chosen
uniformly.
Since the graph is undirected and $d$-regular, we have
\[
P_{vw}=P_{wv}\geq0,
\qquad
\sum_{w\in\mathcal V}P_{vw}=1,
\qquad
\sum_{v\in\mathcal V}P_{vw}=1.
\]
Thus $P$ is symmetric and doubly stochastic.
In particular, $P$ is reversible with respect to the
uniform probability measure $\pi(v):=1/k$ on $\mathcal V$,
meaning that
\[
\pi(v)P_{vw}=\pi(w)P_{wv},
\qquad v,w\in\mathcal V.
\]
It is well known that the eigenvalues of $P$, listed with
multiplicity in nonincreasing order, satisfy
\[
1=\lambda_1(P)>\lambda_2(P)\geq\cdots
\geq\lambda_k(P)\geq-1.
\]
The strict inequality $\lambda_1(P)>\lambda_2(P)$ follows
from the connectedness of $\mathcal G$.
The number
\[
1-\lambda_2(P)
\]
is called the spectral gap of $P$.
Here the eigenvalues are ordered by their numerical values,
rather than their absolute values.

For a fixed integer $d\geq3$, a sequence
$\{\mathcal G_q\}_{q\in\mathbb N}$ of finite connected
$d$-regular graphs is called a family of expanders if,
writing $\mathcal V_q$ for the vertex set,
$k_q:=|\mathcal V_q|$, and $P_q$ for the normalized
adjacency matrix, we have
\[
k_q\longrightarrow\infty
\quad\text{as }q\to\infty,
\quad \text{and} \quad 
\inf_{q\in\mathbb N}
\bigl(1-\lambda_2(P_q)\bigr)>0.
\]
Thus the vertex degrees are fixed, while the spectral gaps
remain bounded below by a positive constant as the numbers
of vertices tend to infinity.

A graph is bipartite if its vertex set can be partitioned
into two disjoint sets such that every edge joins a vertex
in one set to a vertex in the other.
For a connected $d$-regular graph, the adjacency eigenvalue
$d$ is called trivial.
If the graph is bipartite, $-d$ is also called trivial.
All other adjacency eigenvalues are called nontrivial.
The graph is called Ramanujan if every nontrivial
adjacency eigenvalue $\lambda$ satisfies
\[
|\lambda|\leq2\sqrt{d-1}.
\]
By \cite[Theorem 5.5]{MarcusSpielmanSrivastava2015},
for every integer $d\geq3$ there exists a sequence of finite
connected $d$-regular bipartite Ramanujan graphs
$\{\mathcal G_q\}_{q\in\mathbb N}$ with $k_q\to\infty$.
Next, fix \(d\geq 3\) and choose a sequence with the above properties. Let \(A_q\) denote the adjacency matrix of the \(q\)-th graph. Then
\[
\lambda_2(P_q)
=\frac{\lambda_2(A_q)}{d}
\leq \frac{2\sqrt{d-1}}{d}.
\]
Thus this sequence forms a family of expanders with the uniform spectral-gap bound
\[
1-\lambda_2(P_q)\geq
\delta:=1-\frac{2\sqrt{d-1}}d>0,
\qquad q\in\mathbb N.
\]
The positivity follows from
\(d^2-4(d-1)=(d-2)^2>0\).

The spectral gap controls how any Hilbert-space-valued function changes on the vertices. Now fix \(q\in\mathbb N\), let \((\mathscr H,\|\cdot\|)\) be an arbitrary real Hilbert space, and assign a vector \(\xi_v\in\mathscr H\) to each vertex \(v\in\mathcal V_q\). We define the mean of this family by
\[
\bar\xi:=\frac1{k_q}\sum_{v\in\mathcal V_q}\xi_v.
\]
The variance with respect to the uniform probability measure on \(\mathcal V_q\) is
\[
\frac1{k_q}\sum_{v\in\mathcal V_q}
\|\xi_v-\bar\xi\|^2.
\]
The Hilbert-valued spectral-gap inequality gives
\[
\frac{1-\lambda_2(P_q)}{k_q^2}
\sum_{v,w\in\mathcal V_q}\|\xi_v-\xi_w\|^2
\leq
\frac1{k_q}\sum_{v,w\in\mathcal V_q}
(P_q)_{vw}\|\xi_v-\xi_w\|^2;
\]
see \cite[Section~1.1, formula~(4)]{Naor2021}. We use the identity
\[
\frac1{k_q^2}
\sum_{v,w\in\mathcal V_q}\|\xi_v-\xi_w\|^2
=
\frac2{k_q}
\sum_{v\in\mathcal V_q}\|\xi_v-\bar\xi\|^2
\]
and the bound \(1-\lambda_2(P_q)\geq\delta\) to obtain the Poincar\'e inequality
\[
\frac{2\delta}{k_q}
\sum_{v\in\mathcal V_q}\|\xi_v-\bar\xi\|^2
\leq
\frac1{k_q}\sum_{v,w\in\mathcal V_q}
(P_q)_{vw}\|\xi_v-\xi_w\|^2.
\]
The sum on the right is taken over ordered pairs of
vertices.
Only adjacent pairs contribute, and their total
normalized weight is
\[
\frac1{k_q}\sum_{v,w\in\mathcal V_q}(P_q)_{vw}=1.
\]
Thus the right-hand side is the average squared difference
between the vectors assigned to the endpoints of a
uniformly chosen oriented edge.
The inequality bounds the variance of the entire family
by the average squared difference along edges, with
constant $1/(2\delta)$, independently of $q$.
For the Ramanujan graphs chosen above, this constant
depends only on $d$.

The preceding discussion yields the following lemma.

\begin{lem}\label{lem:hilbert-poincare}
	Let $d\geq3$ be an integer, and set
	$\delta:=1-2\sqrt{d-1}/d>0$.
	There exists a sequence $\{\mathcal G_q\}_{q\in\mathbb N}$
	of finite connected simple undirected $d$-regular graphs,
	with vertex sets $\mathcal V_q$ and normalized adjacency
	matrices $P_q$, such that
	$k_q:=|\mathcal V_q|\to\infty$ as $q\to\infty$
	and the following holds:
	for every $q\in\mathbb N$, every real Hilbert space
	$(\mathscr H,\|\cdot\|)$, and every family
	$\{\xi_v\}_{v\in\mathcal V_q}$  of vectors in $\mathscr H$,
	we have
	\begin{equation}\label{eq:poincare}
		\frac{2\delta}{k_q}
		\sum_{v\in\mathcal V_q}\|\xi_v-\bar\xi\|^2
		\leq
		\frac1{k_q}
		\sum_{v,w\in\mathcal V_q}
		(P_q)_{vw}\|\xi_v-\xi_w\|^2,
		\qquad
		\bar\xi:=\frac1{k_q}\sum_{v\in\mathcal V_q}\xi_v.
	\end{equation}
\end{lem}

\section{Main result}

For a real normed space $\mathscr{X}$, 
we write
\[
B_{\mathscr{X}}:=\{x\in\mathscr{X}:\|x\|\leq1\}.
\]
Before proving our main result, we need a lemma.

\begin{lem}\label{lem:centering}
	Let $\mathscr{X}$ be an $n$-dimensional real normed space,
	where $n\geq2$, and let
	\[
	\phi:S(\mathscr{X})\longrightarrow  S(\ell_n^2)
	\]
	be a continuous map that is not null-homotopic. Define
	\[
	\widetilde{\phi}:\mathscr{X}\longrightarrow B_{\ell_n^2},
	\qquad
	\widetilde{\phi}(x):=
	\begin{cases}
		\min\{\|x\|,1\}\,
		\phi \Big( \dfrac{x}{\|x\|} \Big), &x\neq0,\\
		0,&x=0.
	\end{cases}
	\]
	Then, for every integer $k\geq1$ and every
	$x_1,\ldots,x_k\in\mathscr{X}$, there exists
	$a\in\mathscr{X}$ such that
	\[
	\frac1k\sum_{i=1}^k\widetilde{\phi}(a+x_i)=0.
	\]
\end{lem}

\begin{proof}
	The map $\widetilde{\phi}$ is continuous, including at the origin,
	since
	\[
	\|\widetilde{\phi}(x)\|=\min\{\|x\|,1\}.
	\]
	Put
	\[
	A(y):=\frac1k\sum_{i=1}^k\widetilde{\phi}(y+x_i),
	\qquad
	M:=\max_{1\leq i\leq k}\|x_i\|.
	\]
	Suppose, towards a contradiction, that $A(y)\neq0$ for every
	$y\in\mathscr{X}$.
	
	We use the elementary normalization estimate
	\begin{equation}\label{eq:normalization}
		\left\|
		\frac{x}{\|x\|}
		-\frac{y}{\|y\|}
		\right\|
		\leq
		\frac{2\|x-y\|}
		{\max\{\|x\|,\|y\|\}},
		\qquad x,y\in\mathscr{X}\setminus\{0\}.
	\end{equation}

	If $R>M+1$ and $v\in S(\mathscr{X})$, then
	$\|Rv+x_i\|>1$ for all $i$, so
	\[
	\widetilde{\phi}(Rv+x_i)
	=\phi\!\left(\frac{Rv+x_i}{\|Rv+x_i\|}\right).
	\]
	Moreover, \eqref{eq:normalization} gives
	\[
	\left\|
	\frac{Rv+x_i}{\|Rv+x_i\|}-v
	\right\|
	\leq\frac{2M}{R}
	\]
	for all $v\in S(\mathscr{X})$.
	Since $S(\mathscr{X})$ is compact and $\phi$ is continuous,
	it follows that
	\[
	\sup_{v\in S(\mathscr{X})}
	\|A(Rv)-\phi(v)\| \longrightarrow0
	\qquad\text{as }R\longrightarrow\infty.
	\]
	Choose $R>M+1$ so large that this supremum is less than $1$.
	The formula
	\[
	H(v,t):=
	\frac{(1-t)\phi(v)+tA(Rv)}
	{\|(1-t)\phi(v)+tA(Rv)\|},
	\qquad (v,t)\in S(\mathscr{X})\times[0,1],
	\]
	defines a continuous homotopy from $\phi$ to
	\[
	v\longmapsto \frac{A(Rv)}{\|A(Rv)\|}.
	\]
	The denominator in the homotopy is nonzero because its argument
	has distance less than $1$ from the unit vector $\phi(v)$.
	On the other hand, the latter map extends continuously to
	$B_{\mathscr{X}}$ by
	\[
	v\longmapsto\frac{A(Rv)}{\|A(Rv)\|},
	\qquad v\in B_{\mathscr{X}}.
	\]
Contracting $B_{\mathscr X}$ to the origin shows that the
restriction of this extension to $S(\mathscr X)$ is
null-homotopic.
Since this restriction is homotopic to $\phi$ via $H$,
it follows that $\phi$ is null-homotopic, a contradiction.
\end{proof}

We are now ready to prove our main theorem. 

\begin{thm}\label{prop:lp-direct-sum-no-property-H}
	Suppose that  $1\leq  s<\infty$.	
	Let  $\{p_j\}_{j\geq1}$  be a sequence in $[1,\infty)$ such that
	$\sup_j p_j=+\infty$.
	For each $j\geq1$, let $(\Omega_j,\Sigma_j,\mu_j)$ be a
	measure space such that $L^{p_j}(\Omega_j,\mu_j)$ is
	infinite-dimensional.
	Then the real Banach space
	\[
	\mathscr{X}:=
	\left(
	\bigoplus_{j=1}^{\infty}L^{p_j}(\Omega_j,\mu_j)
	\right)_{\ell^s}
	\]
	does not have Property $(H)$.
\end{thm}
\begin{proof}
	Suppose that $\mathscr X$ has Property~$(H)$, and let
	$\phi:S(\mathscr X)\to S(\mathscr H)$ be a corresponding
	uniformly continuous map, where $\mathscr H$ is a real
	Hilbert space.
	By definition, there exists an increasing sequence
	$\{\mathscr V_n\}_{n\geq1}$ of finite-dimensional subspaces
	of $\mathscr X$ such that
	$\bigcup_{n\geq1}\mathscr V_n$ is dense in $\mathscr X$
	and, for every $n$, the restriction
	\[
	\phi_n:=\phi\big|_{S(\mathscr V_n)}:
	S(\mathscr V_n)\longrightarrow S(\ell^2_{N_n}),
	\qquad N_n:=\dim\mathscr V_n,
	\]
	is a homeomorphism.
	Here we use compatible isometric identifications of the
	nested finite-dimensional Hilbert spaces with the coordinate
	subspaces $\ell^2_{N_n}$ of $\ell^2(\mathbb N)$.
	Discarding finitely many initial terms, we may assume that
	$N_n\geq2$ for every $n$.
	
	For $t\geq0$, define
	\[
	\omega_\phi(t):=
	\sup\left\{
	\|\phi(x)-\phi(y)\|:
	x,y\in S(\mathscr X),\
	\|x-y\|_{\mathscr X}\leq t
	\right\}
	\]
	and
	\[
	\omega_{\phi_n}(t):=
	\sup\left\{
	\|\phi_n(x)-\phi_n(y)\|:
	x,y\in S(\mathscr V_n),\
	\|x-y\|_{\mathscr X}\leq t
	\right\}.
	\]
	These functions are nondecreasing, and
	\[
	0\leq\omega_{\phi_n}(t)\leq\omega_\phi(t)\leq2
	\qquad(n\geq1,\ t\geq0).
	\]
	Since $\phi$ is uniformly continuous,
	\[
	\lim_{t\rightarrow0}\omega_\phi(t)=0.
	\]

	\medskip
	\noindent\emph{Uniform copies of finite-dimensional $\ell^\infty$
		inside the paving spaces.}
	Fix an integer $k\geq2$, and choose $m$ such that
	$p_m\geq\log_2 k$.
	Since $L^{p_m}(\Omega_m,\mu_m)$ is infinite-dimensional,
	we can choose functions $f_1,\ldots,f_k\in L^{p_m}(\Omega_m,\mu_m)$
	with pairwise disjoint supports and
	$\|f_i\|_{L^{p_m}(\Omega_m,\mu_m)}=1$ for $i=1,\ldots,k$.
	Let
	\[
	T_k:\ell^\infty_k\longrightarrow\mathscr{X}
	\]
	be the linear map defined as follows. 
	Write an element of $\mathscr{X}$ as
	\[
	x=(x^{(j)})_{j\geq1},
	\qquad
	x^{(j)}\in L^{p_j}(\Omega_j,\mu_j).
	\]
	For $y=(c_1,\ldots,c_k)\in\ell^\infty_k$, define
	\[
	T_k(y):=(x^{(j)})_{j\geq1},
	\]
	where
	\[
	x^{(m)}:=\sum_{i=1}^k c_i f_i
	\in L^{p_m}(\Omega_m,\mu_m),
	\]
	and  all the remaining components are zero.
	
	Since the functions $f_i$ have pairwise disjoint supports
	and unit norm, for every $y\in\ell^\infty_k$ we have
	\[
	\|y\|_\infty
	\leq
	\|T_k(y)\|_{\mathscr{X}}
	=
	\left(\sum_{i=1}^k|c_i|^{p_m}\right)^{1/p_m}
	\leq
	k^{1/p_m}\|y\|_\infty
	\leq
	2\|y\|_\infty.
	\]

	By density and the fact that the spaces $\mathscr{V}_n$ are increasing, we can
	choose $n$ and $x_1,\ldots,x_k\in\mathscr{V}_n$ such that
	\[
	\sum_{i=1}^k\|x_i-T_ke_i\|_{\mathscr{X}}<\frac12,
	\]
	where $e_1,\ldots,e_k$ are the standard coordinate vectors of  $\ell^\infty_k$.
	Define
	\[
	\tau_k:\ell^\infty_k\longrightarrow\mathscr{V}_n,
	\qquad
	\tau_k(y):=2\sum_{i=1}^kc_ix_i,
	\]
	where $y=(c_1,\ldots,c_k)\in\ell^\infty_k$.
	Since
	\[
	\begin{aligned}
		\left\|\sum_{i=1}^k c_ix_i-T_k(y)\right\|_{\mathscr{X}}
		&\leq
		\sum_{i=1}^k |c_i| \cdot \|x_i-T_k(e_i)\|_{\mathscr{X}}\\
		&\leq
		\|y\|_\infty\sum_{i=1}^k\|x_i-T_k(e_i)\|_{\mathscr{X}}\\
		&\leq\frac12\|y\|_\infty,
	\end{aligned}
	\]
	we obtain
	\begin{equation}\label{eq:uniform-cubes}
		\|y\|_\infty\leq\|\tau_k(y)\|_{\mathscr{X}}\leq5\|y\|_\infty,
		\qquad y\in\ell^\infty_k.
	\end{equation}
	In particular, $\tau_k$ is injective and $\dim\mathscr{V}_n\geq k$.

	Fix an integer $d\geq3$, and set
	\[
	\delta:=1-\frac{2\sqrt{d-1}}d>0.
	\]
	Choose a sequence $\{\mathcal G_q\}_{q\in\mathbb N}$
	as in Lemma~\ref{lem:hilbert-poincare}.
	For each $q\in\mathbb N$, let $\mathcal V_q$ be its vertex set,
	let $P_q$ be its normalized adjacency matrix, and put
	$k_q:=|\mathcal V_q|$.
	Then $k_q\to\infty$ as $q\to\infty$, and
	\eqref{eq:poincare} holds for every family of vectors in
	any real Hilbert space.

	\medskip
	\noindent\emph{Embedding the graphs and centering the sphere maps.}
	For each $q\in\mathbb{N}$, 
	use \eqref{eq:uniform-cubes} with $k=k_q$ to obtain
	\[
	\tau_{k_q}:\ell^\infty_{k_q}
	\longrightarrow\mathscr{V}_{n_q},
	\qquad
	\|y\|_\infty
	\leq\|\tau_{k_q}y\|_{\mathscr{X}}
	\leq5\|y\|_\infty.
	\]
	Fix an ordering of $\mathcal V_q$ and use it to identify
	$\ell^\infty(\mathcal V_q)$ with $\ell^\infty_{k_q}$.
	Let $d_q$ denote the shortest-path metric on $\mathcal{G}_q$,
	and define
	\[
	\psi_q:\mathcal V_q\longrightarrow\ell^\infty_{k_q},
	\qquad
	\psi_q(v):=\bigl(d_q(v,w)\bigr)_{w\in\mathcal V_q}.
	\]
	The triangle inequality, together with the coordinate indexed by $v$,
	gives
	\[
	\|\psi_q(v)-\psi_q(w)\|_\infty=d_q(v,w).
	\]
	Fix an integer $r\geq1$, and put
	\[
	x_v:=\frac1r\tau_{k_q}\bigl(\psi_q(v)\bigr),
	\qquad v\in\mathcal V_q.
	\]
	Then
	\begin{equation}\label{eq:graph-distances}
		\frac1r\,d_q(v,w)
		\leq\|x_v-x_w\|_{\mathscr{X}}
		\leq\frac5r\,d_q(v,w).
	\end{equation}
	
	Let $\widetilde\phi_{n_q}$ be the extension of $\phi_{n_q}$
	defined in Lemma~\ref{lem:centering}.
	Since
	\[
	N_{n_q}=\dim\mathscr V_{n_q}\geq k_q\geq 2,
	\]
	the map \(\phi_{n_q}\) is not null-homotopic, as it is a homeomorphism between the corresponding unit spheres. By Lemma~\ref{lem:centering}, there exists
	\(a=a_{q,r}\in\mathscr V_{n_q}\) such that the vectors
	\[
	\xi_v:=\widetilde\phi_{n_q}(a+x_v)
	\in\ell^2_{N_{n_q}}
	\]
	satisfy
	\begin{equation}\label{eq:zero-mean}
		\frac1{k_q}\sum_{v\in\mathcal V_q}\xi_v=0.
	\end{equation}
	
	\medskip
	\noindent\emph{Estimates for the variance and edge energy.}
	Define
	\[
	\mathcal B_{q,r}
	:=\{v\in\mathcal V_q:\|a+x_v\|_{\mathscr{X}}<1\}.
	\]
	If this set is nonempty, fix $v_0\in\mathcal B_{q,r}$.
	For every $v\in\mathcal B_{q,r}$,
	\eqref{eq:graph-distances} yields
	\[
	\frac1r\,d_q(v,v_0)
	\leq\|x_v-x_{v_0}\|_{\mathscr{X}}
	\leq
	\|a+x_v\|_{\mathscr{X}}
	+\|a+x_{v_0}\|_{\mathscr{X}}
	<2.
	\]
	Consequently, $\mathcal B_{q,r}$ is contained in a graph ball of
	radius $2r$. Since every vertex has degree $d$,
	\begin{equation}\label{eq:bad-vertices}
		|\mathcal B_{q,r}|
		\leq C_{d,r}:=\sum_{j=0}^{2r}d^j.
	\end{equation}
	This also holds if $\mathcal B_{q,r}$ is empty.
	The constant $C_{d,r}$ is independent of $q$.

	By the definition of $\widetilde\phi_{n_q}$, we have
	$\|\xi_v\|_2\leq1$ for all $v$ and
	$\|\xi_v\|_2=1$ for $v\notin\mathcal B_{q,r}$. Hence
	\begin{equation}\label{eq:variance-lower}
		\frac1{k_q}\sum_{v\in\mathcal V_q}\|\xi_v\|_2^2
		\geq1-\frac{C_{d,r}}{k_q}.
	\end{equation}
	
	Suppose that $v,w$ are adjacent and neither belongs to
	$\mathcal B_{q,r}$.
	Then
	$\|a+x_v\|_{\mathscr{X}},
	\|a+x_w\|_{\mathscr{X}}\geq1$
	and $\|x_v-x_w\|_{\mathscr{X}}\leq5/r$.
	Using \eqref{eq:normalization} and \eqref{eq:graph-distances}, we get
	\begin{align*}
		\|\xi_v-\xi_w\|_2
		&=
		\left\|
		\phi_{n_q}\!\left(
		\frac{a+x_v}{\|a+x_v\|_{\mathscr{X}}}
		\right)
		-
		\phi_{n_q}\!\left(
		\frac{a+x_w}{\|a+x_w\|_{\mathscr{X}}}
		\right)
		\right\|_2\\
		&\leq\omega_{\phi_{n_q}}(10/r).
	\end{align*}
	For the remaining adjacent pairs, the bound
	$\|\xi_v-\xi_w\|_2^2\leq4$ suffices.
	Since $P_q$ is symmetric and stochastic,
	\begin{align*}
		\frac1{k_q}
		\sum_{\substack{v,w\in\mathcal V_q\\
				v\in\mathcal B_{q,r}\ \text{or}\
				w\in\mathcal B_{q,r}}}
		(P_q)_{vw}
		&\leq
		\frac1{k_q}
		\sum_{v\in\mathcal B_{q,r}}\sum_w(P_q)_{vw}
		+
		\frac1{k_q}
		\sum_{w\in\mathcal B_{q,r}}\sum_v(P_q)_{vw}\\
		&=\frac{2|\mathcal B_{q,r}|}{k_q}.
	\end{align*}
	It follows that
	\begin{equation}\label{eq:edge-upper}
		\frac1{k_q}\sum_{v,w\in\mathcal V_q}
		(P_q)_{vw}\|\xi_v-\xi_w\|_2^2
		\leq
		\omega_{\phi_{n_q}}(10/r)^2
		+\frac{8C_{d,r}}{k_q}.
	\end{equation}
	
	Applying Lemma~\ref{lem:hilbert-poincare} to the family
	$\{\xi_v\}_{v\in\mathcal V_q}$ in $\ell^2_{N_{n_q}}$,
	and using \eqref{eq:zero-mean}, \eqref{eq:variance-lower},
	and \eqref{eq:edge-upper}, we obtain
	\[
	2\delta\left(1-\frac{C_{d,r}}{k_q}\right)
	\leq
	\omega_{\phi_{n_q}}(10/r)^2
	+\frac{8C_{d,r}}{k_q},
	\]
	or equivalently,
	\[
	\omega_{\phi_{n_q}}(10/r)^2
	\geq
	2\delta-(2\delta+8)\frac{C_{d,r}}{k_q}.
	\]
	Consequently,
	\[
	\omega_{\phi}(10/r)^2
	\geq
	2\delta-(2\delta+8)\frac{C_{d,r}}{k_q}.
	\]
	For each fixed $r$, let $q\to\infty$.
	Since $k_q\to\infty$ and $C_{d,r}$ is independent of $q$,
	\[
	\omega_{\phi}(10/r)^2\geq2\delta.
	\]
	Now let $r\to\infty$.
	Since $\phi$ is uniformly continuous,
	$\omega_\phi(10/r)\to0$.
	This contradicts the inequality
	$\omega_\phi(10/r)^2\geq2\delta$, since $\delta>0$.
	Therefore $\mathscr X$ does not have Property~$(H)$.
\end{proof}

\begin{rmk}
	The same argument applies to direct sums of Schatten $p$-classes.
	Let $1\leq s<\infty$, and let
	$\{p_j\}_{j\geq1}\subset[1,\infty)$ satisfy
	$\sup_j p_j=\infty$.
	For each $j\geq1$, let $\mathscr H_j$ be an
	infinite-dimensional real or complex Hilbert space.
	Denote by $S^p(\mathscr H_j)$ the Schatten $p$-class of
	compact operators on $\mathscr H_j$, equipped with the norm
	\[
	\|T\|_{S^p}
	:=
	\left(\operatorname{Tr}\bigl((T^*T)^{p/2}\bigr)\right)^{1/p}.
	\]
	Then
	\[
	\mathscr X:=
	\left(
	\bigoplus_{j=1}^{\infty}S^{p_j}(\mathscr H_j)
	\right)_{\ell^s},
	\]
	regarded as a real Banach space, does not have Property~$(H)$.
	
	Indeed, fix $k\geq2$ and choose $m$ such that
	$p_m\geq\log_2 k$.
	Choose rank-one orthogonal projections
	$E_1,\ldots,E_k$ on $\mathscr H_m$ with pairwise
	orthogonal ranges.
	For $y=(c_1,\ldots,c_k)\in\ell^\infty_k$, we have
	\[
	\left\|\sum_{i=1}^k c_iE_i\right\|_{S^{p_m}}
	=
	\left(\sum_{i=1}^k|c_i|^{p_m}\right)^{1/p_m}.
	\]
	Placing this operator in the $m$-th summand and taking
	all the remaining components to be zero defines a linear map
	$T_k:\ell^\infty_k\to\mathscr X$ satisfying
	\[
	\|y\|_\infty
	\leq\|T_k(y)\|_{\mathscr X}
	\leq k^{1/p_m}\|y\|_\infty
	\leq2\|y\|_\infty.
	\]
	The remainder of the proof of
	Theorem~\ref{prop:lp-direct-sum-no-property-H}
	then applies without change.
\end{rmk}
\section*{AI Declaration}

AI provided inspiration for this work, but the authors independently developed and proved the theorem. AI was also used to improve the English writing. The authors take full responsibility for the content of this paper.

\end{document}